\documentclass[11pt,reqno]{amsart}

\usepackage[tmargin=1.2in,bmargin=1.2in,rmargin=1.2in,lmargin=1.2in]{geometry}

\usepackage{tikz}

\usepackage[breaklinks=true]{hyperref}

\usepackage{enumerate}

\usepackage{amsmath,amsthm,amsfonts,amssymb,amscd,color,graphicx}
\usepackage{mathrsfs,tocvsec2}

\theoremstyle{plain}
\newtheorem{theorem}{Theorem}[section]
\newtheorem{lemma}[theorem]{Lemma}

\newtheorem{cor}[theorem]{Corollary}
\newtheorem{utheorem}{\textrm{\textbf{Theorem}}}

\theoremstyle{definition}
\newtheorem{defn}[theorem]{Definition}
\newtheorem{remark}[theorem]{Remark}
\newtheorem{example}[theorem]{Example}
\newtheorem{question}[theorem]{Question}

\numberwithin{equation}{section}

\newcommand{\I}{\mathrm{I}}

\newcommand{\F}{\mathbb{F}}
\newcommand{\R}{\mathbb{R}}

\newcommand{\D}{\mathbb{D}}

\newcommand{\floo}[1]{\lfloor #1 \rfloor}

\renewcommand{\geq}{\geqslant}
\renewcommand{\leq}{\leqslant}

\begin{document}

\title[Cholesky decomposition for symmetric matrices over finite fields]{Cholesky decomposition for symmetric matrices\\ over finite fields}


\author[Prateek Kumar Vishwakarma]{Prateek Kumar Vishwakarma}
\address[]{D\'epartement de math\'ematiques et de
statistique, Universit\'e Laval, Qu\'ebec, Canada}
\email{\tt prateek-kumar.vishwakarma.1@ulaval.ca,
prateekv@alum.iisc.ac.in}

\begin{abstract}
Inspired by the seminal work of Andr\'e-Louis Cholesky -- whose contributions remain crucial in broader sciences even after more than a century -- Cooper, Hanna and Whitlatch (2024) developed a theory of positive matrices over finite fields, and Khare and Vishwakarma (2025) described a general Cholesky factorization for a dense sub-family of the cone of Hermitian matrices over real/complex fields, whose leading principal minors (LPM) are nonzero. Building on this, we develop a parallel theory within the finite field setting. Specifically \textit{(i)} we extend the general Cholesky factorization to the LPM cone over finite fields which has asymptotic density $1$. We show that \textit{(ii)}
this factorization is compatible with the entrywise Frobenius map, recently studied in the context of positivity preservers by Guillot, Gupta, Vishwakarma, and Yip [\textit{J.\ Algebra}, 2025]. We also \textit{(iii)} leverage the Cholesky-structures to define meaningful group operations on the matrix cone, and as an application \textit{(iv)} enumerate sub-cones of LPM matrices using our general Cholesky factorizations.
\end{abstract}

\subjclass[2020]{
15B48, 
15A23, 
15B33}

\keywords{Finite field, 
Cholesky decomposition,
Cholesky factorization,
sign pattern,
LPM matrix,
TPM matrix, Frobenius map}

\date{\today}

\maketitle


\section{Introduction and main results}

For an integer $n\geq 1$, over the real or complex field $\F$, a Hermitian matrix $A\in \F^{n\times n}$ is called positive definite if the quadratic form $z^*Az>0$ for all nonzero $z\in \F^{n}$. Among more than half a dozen equivalent definitions of these matrices, the one given by Andr\'e-Louis Cholesky's Factorization Theorem stands out for its applicability in the broader sciences and mathematics, even after one hundred years of its publication in 1924~\cite{Benoit}. While the applicability of Cholesky's factorization of the positive definite cone of matrices has extensively been explored for over several decades, there have been fewer attempts made to extend it to a broader class of matrices, and for fields that are not necessarily real or complex. We explore this extension over finite fields, and we begin by mentioning two such recent works \cite{cooper2024positive, khare2025cholesky} that inspired the current article in such an algebraic setting.

In a recent work by Cooper, Hanna, and Whitlatch \cite{cooper2024positive}, the cone of positive definite matrices over a finite field $\F_q$ (with $q$ elements) is introduced. The authors define a field element to be positive if it is a nonzero square in the field. Then, a symmetric matrix with entries in $\F_q$ is called positive definite if all its leading principal minors are positive -- that is, they are squares in the field. The authors show that such matrices admit a Cholesky-type factorization:

\begin{theorem}[\cite{cooper2024positive}]\label{T:CHW}
Only when $q$ is even or $q \equiv 3 \pmod{4}$, a symmetric matrix $A$ with entries in $\F_q$ is positive definite if and only if there exists a unique lower triangular matrix $L$ with positive diagonal entries such that $A = LL^T$.    
\end{theorem}

While this result establishes a compelling analogue of the Cholesky factorization in the finite field setting, other characterizations of positive definiteness familiar from the real or complex cases do not yet appear to have meaningful analogues over finite fields. Nonetheless, the existence of a Cholesky factorization in the finite field context is leveraged in \cite{cooper2024positive} to demonstrate the existence of certain pressing sequences for weighted graphs. Furthermore, this work has served as a key motivation for developing a theory of positivity-preserving entrywise transformations over finite fields \cite{guillot2024positivity,guillot2024positivity-fpsac}, which we will discuss in more detail in a subsection later.

The second relevant work is more recent and is by Khare and Vishwakarma \cite{khare2025cholesky}, who introduce a general Cholesky factorization and develop a rich theory for a dense open set of real/complex Hermitian matrices. Their framework decomposes the space into cones of matrices determined by the sign patterns of their leading principal minors. Additionally, the framework explores several implications of this decomposition.\medskip

We adopt an idea from \cite{khare2025cholesky} and explore the following points (and others discussed later):

\begin{enumerate}[$(a)$]
    \item We extend Theorem~\ref{T:CHW}, which addresses positive definite matrices via their Cholesky decomposition, to broader cones of symmetric matrices over a finite field. 
    \item The general Cholesky decomposition discussed in \cite{khare2025cholesky} is done for fields (including the real-closed fields) in which the underlying/base field is totally ordered. In the current work over finite fields, the Cholesky decomposition is discussed over fields that do not have a total order, thus demonstrating such factorizations in a different algebraic setting.
    \item We explore implications that arise specifically in the finite field context, particularly in light of recent developments regarding entrywise transforms in this direction in \cite{guillot2024positivity}. 
\end{enumerate}

Let us now introduce the new matrix cones over a finite field.

\subsection{Notations} Throughout this article, we let $\F_q$ denote the finite field with $q = p^k$ elements, where $p$ is a prime and $k \geq 1$ an integer. The multiplicative cyclic group of nonzero elements of $\F_q$ is denoted by $\F_q^\times$. An element \( a \in \F_q^\times \) is called \emph{positive} if there exists \( b \in \F_q^{\times} \) such that \( a = b^2 \). We denote the set of all positive elements by \( \F_q^+ \), and define the set of \emph{negative} elements as \( \F_q^- := \F_q^\times \setminus \F_q^+ \). We distinguish these elements using the \emph{quadratic/sign character}
\begin{align*}
 \chi:\F_q^\times \to \{\pm 1\}  \quad \mbox{defined by}\quad \chi(a) :=
\begin{cases}
\phantom{-}1 & \text{if } a \in \F_q^+, \\
-1 & \text{if } a \in \F_q^-,
\end{cases}
\end{align*}
extended multiplicatively to all of \( \F_q \) by setting \( \chi(0) := 0 \). It is well known that if \( q \) is odd, then \( -1 \in \F_q^{-} \) if and only if \( q \equiv 3 \pmod{4} \) if and only if $\F_q^{-}=-\F_q^+$ {\cite[Proposition~2.3]{guillot2024positivity}}. In this case, the field \( \F_q \) is referred to as \emph{definite}; otherwise, when \( q \equiv 1 \pmod{4} \), it is called \emph{non-definite}. Finally, we use $\I_n$ to denote the $n\times n$ identity matrix over the field in context, and $A^{-T}=(A^{-1})^T$ for invertible matrices.

\begin{defn}[Matrix cones with sign patterns]\label{defn:LPM} Given an integer \( n \geq 1 \) and a sign pattern \( \epsilon \in \{\pm1\}^n\subseteq \R^{n} \), we define \( LPM_n^{\F_q}(\epsilon) \) to be the cone of symmetric matrices \( A \in \F_q^{n \times n} \) such that, for each \( k = 1, \dots, n \), the \( k \times k \) leading principal minor of \( A \) has quadratic character \( \epsilon_k \).

Similarly define \( TPM_n^{\F_q}(\epsilon) \) as the cone of symmetric matrices \( A \in \F_q^{n \times n} \) with all its $k \times k$ \textit{trailing} principal minors having quadratic characters \( \epsilon_k \).
\end{defn}

\noindent Each $LPM_n^{\F_q}(\epsilon)$ is nonempty, and so is $TPM_n^{\F_q}(\epsilon)$: define $\nabla,{\mathbb D}_\epsilon(\omega_{\pm})\in \F_q^{n\times n}$ via,
\begin{equation}\label{Elpmdiag}
{\mathbb D}_\epsilon(\omega_{\pm}) := 
\begin{pmatrix}
\omega_1 & 0 & 0 & \cdots & 0\\
0 & \omega_1 \omega_2 & 0 & \cdots & 0\\
0 & 0 & \omega_2 \omega_3 & \cdots & 0\\
\vdots & \vdots & \vdots & \ddots & \vdots\\
0 & 0 & 0 & \cdots & \omega_{n-1} \omega_n
\end{pmatrix} \quad \mbox{and}\quad \nabla:= 
\begin{pmatrix}
0 & 0  & \cdots & 1\\
\vdots & \vdots & \ddots & \vdots\\
0 & 1  & \cdots & 0\\
1 & 0  & \cdots & 0
\end{pmatrix},
\end{equation}
where each $\omega_k\in \{\omega_+,\omega_-\}$ for some fixed $\omega_{+}\in \F_q^{+}$ and $\omega_-\in \F_q^{-}$ such that $\chi(\omega_k)=\epsilon_k$. Then 
\begin{align}
\D_{\epsilon}(\omega_{\pm})\in LPM_n^{\F_q}(\epsilon) \quad \mbox{and}\quad    \nabla \cdot \D_{\epsilon}(\omega_{\pm}) \cdot \nabla\in TPM_n^{\F_q}(\epsilon).
\end{align}
The cones $LPM_n^{\F_q}$ and $TPM_n^{\F_q}$, defined as the sets of all symmetric matrices with nonzero leading, and respectively trailing, principal minors, decompose into disjoint unions:
\begin{align}
LPM_n^{\F_q}=\bigsqcup_{\epsilon\in \{\pm 1\}^{n}} LPM_n^{\F_q}(\epsilon) \quad \mbox{and}\quad TPM_n^{\F_q}=\bigsqcup_{\epsilon\in \{\pm 1\}^{n}} TPM_n^{\F_q}(\epsilon),    
\end{align}
provided \( q \) is odd.
And when $q$ is even, each $LPM_n^{\F_q}(\epsilon)=LPM_n^{\F_q}$ and $TPM_n^{\F_q}(\epsilon)=TPM_n^{\F_q}$.

\subsection{Cholesky decomposition and enumerations}

Our first main result addresses the Cholesky decompositions for the LPM and TPM matrix cones over a finite field.

\begin{utheorem}[Cholesky decomposition over finite fields]\label{main-thm-1}
Fix an integer $n \geq 1$ and a finite field $\F_q$. Suppose a sign pattern $\epsilon \in \{ \pm 1 \}^n\subseteq \R^n$ and a matrix
$A_\epsilon \in LPM_n^{\F_q}(\epsilon)$ are given. Then:

\begin{enumerate}
    \item For each invertible lower triangular $L\in \F_q^{n\times n}$, the matrix $LA_{\epsilon}L^T\in LPM_n^{\F_q}(\epsilon)$. 
    \item In particular, if $\F_q$ is a definite field or has characteristic 2, then for each $A\in LPM_n^{\F_q}(\epsilon)$, there exists a unique lower triangular matrix $L\in \F_q^{n\times n}$ with positive diagonal entries such that $A=LA_{\epsilon}L^T$. 
\end{enumerate}
Moreover, the following linear and nonlinear transforms are bijections:
\begin{align*}
LPM_n^{\F_q}(\epsilon)\to TPM_n^{\F_q}(\epsilon')\quad \mbox{defined by}\quad   A \mapsto \nabla A \nabla, \quad \mbox{and} \quad A\mapsto A^{-1},
\end{align*}
where $\epsilon = \epsilon'$ in the linear case, and $\epsilon'=(\epsilon_n\epsilon_{n-1},\epsilon_n\epsilon_{n-2},\dots,\epsilon_n\epsilon_{1},\epsilon_n)$ in the nonlinear case. 

Therefore, the analogues of (1) and (2) hold for each fixed $A^{\epsilon}\in TPM_{n}^{\F_q}(\epsilon)$: for each invertible upper triangular $U\in \F_q^{n\times n}$, the matrix $UA^{\epsilon}U^T\in TPM_n^{\F_q}(\epsilon)$; and if $\F_q$ is a definite field or has characteristic $2$, then for each $A\in TPM_n^{\F_q}(\epsilon)$, there exists a unique upper triangular matrix $U\in \F_q^{n\times n}$ with positive diagonal entries such that $A=UA^{\epsilon}U^T$.
\end{utheorem}

A couple of remarks are in order:

\begin{remark}[Two-fold refinement in Theorem~\ref{main-thm-1}]\label{Rem:two-fold}
Theorem~\ref{main-thm-1} establishes that for a definite field $\F_q$, the map \( L \mapsto L A_{\epsilon} L^T \) is a bijection from the set of lower triangular matrices with positive diagonal entries to the cone \( {LPM}_n^{\F_q}(\epsilon) \). As such, it constitutes a twofold extension of Theorem~\ref{T:CHW}:
\begin{itemize}
    \item[(a)] it applies to \emph{every} sign pattern \( \epsilon \in \{ \pm 1 \}^n \), including the all-ones vector \( \epsilon = (1,1,\dots,1) \), which corresponds to the positive definite cone considered in Theorem~\ref{T:CHW}; and
    \item[(b)] for each cone \({LPM}_n^{\F_q}(\epsilon) \), including the positive definite cone, it provides \( {LPM}_n^{\F_q}(\epsilon) \)-many Cholesky-type factorizations, one for \emph{each} matrix \( A_{\epsilon} \). For instance, the ``classical'' Cholesky decomposition in Theorem~\ref{T:CHW} arises as the special case when \( A_{\epsilon} = \I_n \), but Theorem~\ref{main-thm-1} allows one to consider distinct Cholesky decompositions for arbitrary choices of positive definite \( A_{(1,\dots,1)} \).
\end{itemize}

Note, both of these extensions apply to a definite field $\F_q$, and exactly one applies in the case where $q$ is even. However, they do not apply to the remaining finite fields (of non-definite type), as such fields do not admit Cholesky factorizations, pointed out in the next remark.
\end{remark}

\begin{remark}[Theorem~\ref{main-thm-1} for non-definite fields]\label{Rem:enumT-via-choleskyT}

Theorem~\ref{main-thm-1}(2) fails to hold over non-definite fields $\F_q$, as every matrix of the form \( L\D_{\epsilon}(\omega_{\pm})L^T \) coincides with \( (LD)\D_{\epsilon}(\omega_{\pm})(LD)^T \) for all diagonal matrices \( D \) with diagonal entries in \( \{ \pm1 \} \subset \F_q \), where \( \D_{\epsilon}(\omega_{\pm}) \) is defined as in~\eqref{Elpmdiag}. Consequently, the uniqueness of the corresponding factorization also breaks down for the TPM cones over non-definite fields.
\end{remark}

As an immediate corollary of Theorem~\ref{main-thm-1}, we provide enumerations of the LPM and TPM (sub-)cones.

\begin{cor}[Enumerations]\label{T:enum}
Fix an integer $n \geq 1$, and let $\F_q$ be a finite field with $q$ odd. Then we have
\[
\# TPM_{n}^{\F_q}=\# LPM_{n}^{\F_q} = 2^n \# LPM_{n}^{\F_q}(\epsilon)=2^n \# TPM_{n}^{\F_q}(\epsilon) =  {(q-1)^n}q^{\binom{n}{2}}
\]
for all sign patterns $\epsilon \in \{\pm 1\}^{n}$, where $q^{\binom{1}{2}}:=1$. Alternately, if $q$ is even, then $$ \# TPM_{n}^{\F_q}= \#LPM_{n}^{\F_q} =  \# LPM_{n}^{\F_q}(\epsilon)= \# TPM_{n}^{\F_q}(\epsilon) =  {(q-1)^n}q^{\binom{n}{2}}.$$
\end{cor}

While this enumerative result can be established directly through inductive arguments, a more elegant and often preferred approach in enumerative combinatorics involves constructing a bijection with a well-structured set for which the counting problem is more tractable (a perspective beautifully illustrated by the numerous celebrated interpretations of the Catalan numbers catalogued by Richard Stanley \cite{stanley}). In the case at hand, rather than pursuing a purely combinatorial classification, we draw upon the structural insights afforded by our general Cholesky factorizations in a finite field setting, which provide a natural (and powerful!) framework for approaching the enumerations.

\begin{proof}[Proof of Corollary~\ref{T:enum}]
For definite fields and for fields with characteristic $2$: one needs to count the number of lower triangular matrices with positive diagonal entries. For the remaining non-definite fields: use induction and the fact that
\begin{align*}
x\mapsto \det \begin{pmatrix}A & {\bf u} \\ {\bf u}^T & x
\end{pmatrix}_{n\times n}\quad \mbox{is a bijection over $\F_q$},
\end{align*}
for a given vector ${\bf u}\in \F^{n-1}$ and nonsingular $A\in \F^{(n-1)\times (n-1)}$.
\end{proof}

\begin{remark}[A disclaimer on the asymptotic density]
Let $\mathrm{Sym}_{n}^{\F_q}$ denotes the set of all $n\times n$ symmetric matrices over a finite field $\F_q$. For a fixed $n\geq 1$, the proportion of matrices in $\mathrm{Sym}_{n}^{\F_q}$ that lie in the cone $LPM_n^{\F_q}$ (and similarly $TPM_n^{\F_q}$) tends to $1$ as $q \to \infty$; that is,
\begin{align*}
    \lim_{q\to \infty} \frac{\# LPM_n^{\F_q}}{\# \mathrm{Sym}_{n}^{\F_q}} = 1.
\end{align*}
It may be worth reiterating that the asymptotic density in question pertains to the setting where the matrix size $n$ is fixed, and the field size $q=p^k$ grows without bound (e.g., the integer $k\to \infty$ if we fix $p$).
\end{remark}

\subsection{The compatibility of Cholesky and Frobenius via Entrywise Transforms}\label{subsec:entrywise}
As we saw in Theorem~\ref{main-thm-1}(1) and (2), that the following map is a bijection for a definite field $\F_q$:
\begin{align*}
\Psi_{A_{\epsilon}}: {\bf L}_{n}^{\F_q^{+}} \to LPM_{n}^{\F_q}(\epsilon) \quad \text{defined by} \quad L \mapsto LA_{\epsilon}L^T,
\end{align*}
for any fixed $A_\epsilon \in LPM_{n}^{\F_q}(\epsilon)$, where ${\bf L}_{n}^{\F_q^{+}}$ denotes the cone of lower triangular matrices with positive diagonal entries in $\F_q$. Consequently, we obtain the bijective composition:
\begin{align*}
\Psi_{A_{\epsilon'} \to A_{\epsilon}}:=\Psi_{A_{\epsilon}} \circ \Psi_{A_{\epsilon'}}^{-1}: LPM_{n}^{\F_q}(\epsilon') \to {\bf L}_{n}^{\F_q^{+}} \to LPM_{n}^{\F_q}(\epsilon)\quad \mbox{given by}\quad LA_{\epsilon'}L^T \mapsto LA_{\epsilon}L^T.
\end{align*}
This provides a Cholesky-based mechanism for transitioning between the cones $LPM_{n}^{\F_q}(\epsilon')$ and $LPM_{n}^{\F_q}(\epsilon)$. Moreover, it prompts the following question: do there exist other maps -- perhaps unrelated to Cholesky factorization and over any given finite field -- that facilitate such transitions between the LPM cones? We address this question for a special class of transformations whose study dates back over a century to Schur -- the student of Frobenius -- including recent developments in the finite field setting.

\begin{defn}[Entrywise Transforms]
Suppose $\F_q$ is a finite field, and consider a map $f:\F_q\to \F_q$. Then this map has a natural extension to all of the matrix space $\F_q^{n\times n}$, for any integer $n\geq 1$ by considering:
\begin{align*}
f[-]:\F_q^{n\times n} \to \F_q^{n\times n}\quad \mbox{defined by}\quad f[A]:=(f(a_{ij}))_{i,j=1}^{n},
\end{align*}
for all $A=(a_{ij})_{i,j=1}^{n}\in \F_q^{n\times n}$. These matrix functions are referred to as Entrywise Transforms.    
\end{defn}

These transforms over finite fields were recently studied by Guillot, Gupta, Vishwakarma, and Yip~\cite{guillot2024positivity-fpsac,guillot2024positivity}, who extended celebrated classical results of Schoenberg~\cite{schoenberg1942positive} and Rudin~\cite{rudin1959positive} to the algebraic framework of finite fields. The foundational work of Schoenberg and Rudin itself traces back to a product theorem \cite{schur1911} of Schur, and to a seminal observation by Pólya and Szegő~\cite{polya-szego}. For a comprehensive account of recent advances in the theory of entrywise transforms, see the monograph by Khare~\cite{khare-book}.

We now bring entrywise transforms to the present setting. We will show that our Cholesky decomposition is compatible with applying the Frobenius entrywise. Namely for any sign pattern $\epsilon\in \{\pm 1\}^{n}\subseteq \R^{n}$, and for special choices of $A_{\epsilon}=\D_{\epsilon}:=\D_{\epsilon}({\pm}1)$ for $\omega_{\pm}=\pm 1$ in \eqref{Elpmdiag},
\begin{equation}
\Psi_{\D_{\epsilon}}^{-1}(\mathrm{Frob}_p[A]) = \mathrm{Frob}_p[\Psi_{\D_{\epsilon}}^{-1}(A)]    \quad \mbox{for all}\quad A\in LPM_{n}^{\F_q}(\epsilon).
\end{equation}

This compatibility question can be pursued within a natural framework. Recall that the map $\Psi_{A_{\epsilon}'\to A_{\epsilon}}$ provides a bijection between the cones $LPM_{n}^{\F_q}(\epsilon')$ and $LPM_{n}^{\F_q}(\epsilon)$. While this map is algorithmic, it is natural to ask whether a simpler, less nuanced map can be obtained in this context. To explore this, we consider entrywise transforms. The following theorem characterizes such entrywise maps over a finite definite field, subject to a certain $2$-$by$-$2$ sign pattern constraint.

\begin{utheorem}[Entrywise transforms and their compatibility with Frobenius]\label{T:entrywise-definite}
Suppose that integers $2\leq s\leq n$ are given, and let $\F_q$ be a finite definite field with $q=p^k$, where $p$ is prime. Let $\epsilon,\epsilon'\in \{\pm 1\}\subseteq \R^{n\times n}$ be sign patterns satisfying $\epsilon_1=\dots=\epsilon_s=\epsilon_1'=\dots=\epsilon_s'=1$. Then for a given function $f:\F_q\to \F_q$, the following are equivalent:
\begin{enumerate}
    \item The entrywise transform $f[-]$ sends $LPM_{n}^{\F_q}(\epsilon')$ into $LPM_{n}^{\F_q}(\epsilon)$.
    \item The sign patterns are equal, and $f$ is a positive multiple of a power of the Frobenius:
    \[
    \epsilon'=\epsilon \qquad \mbox{and}\qquad 
    f\equiv c\cdot \mathrm{Frob}_p^{\ell} \quad \mbox{for some}\quad 0\leq \ell\leq k-1 \mbox{ and }c\in \F_q^+.
    \]
\end{enumerate}
Moreover, every such map $f[-]\equiv c\cdot \mathrm{Frob}_p^{\ell}[-]$ is compatible with Cholesky maps $\Psi_{A_{\epsilon}}^{\pm 1}$ in the following sense: for all $A\in LPM_{n}^{\F_q}(\epsilon)$, $L\in {\bf L}_{n}^{\F_q^+}$ and $0\leq \ell\leq k-1$, we have
\begin{align*}
&\Psi_{A_{\epsilon}}^{-1}(c\cdot \mathrm{Frob}_{p}^{\ell}[A]) = \mathrm{Frob}_p^{\ell}[\Psi_{c\cdot \mathrm{Frob}_p^{\ell}[A_{\epsilon}]}^{-1}(A)] = \sqrt{c}\cdot \mathrm{Frob}_p^{\ell}[\Psi_{\mathrm{Frob}_p^{\ell}[A_{\epsilon}]}^{-1}(A)]\\
\mbox{and}\qquad 
&c\cdot \mathrm{Frob}_p^{\ell}[\Psi_{A_{\epsilon}}(L)] = \Psi_{c\cdot \mathrm{Frob}_p^{\ell}[A_{\epsilon}]}(\mathrm{Frob}_p^{\ell}[L]) = \Psi_{\mathrm{Frob}_p^{\ell}[A_{\epsilon}]}(\sqrt{c}\cdot \mathrm{Frob}_p^{\ell}[L]),
\end{align*}
where $\sqrt{c}$ denotes the unique $d\in \F_q^+$ such that $d^2=c$. 

In particular if $c=1$ and the chosen $A_{\epsilon}=\D_{\epsilon}:=\D_{\epsilon}({\pm}1)$ for $\omega_{\pm}=\pm 1$ as in \eqref{Elpmdiag}, then
\[
\Psi_{\D_{\epsilon}}^{\pm1} \circ \mathrm{Frob}_p = \mathrm{Frob}_p\circ \Psi_{\D_{\epsilon}}^{\pm 1} \qquad \mbox{i.e.,}\qquad 
\Psi_{\D_{\epsilon}}^{\pm 1}(\mathrm{Frob}_p[M_{\pm}]) = \mathrm{Frob}_p[\Psi_{\D_{\epsilon}}^{\pm 1}(M_{\pm})],
\]
for all $M_-\in LPM_{n}^{\F_q}(\epsilon)$ and all $M_{+}\in {\bf L}_{n}^{\F_q^+}$.
\end{utheorem}

\begin{remark}[Two main aspects of Theorem~\ref{T:entrywise-definite}]\label{Rem:entrywise-definite}
\ 
\begin{enumerate}[(a)]
\item Theorem~\ref{T:entrywise-definite} establishes the existence of an entrywise transform \( f[-] \) from \( LPM_{n}^{\F_q}(\epsilon') \) to \( LPM_{n}^{\F_q}(\epsilon) \), under the assumption that the sign patterns \( \epsilon' \) and \( \epsilon \) are identical and take the value \(1\) in at least the first two coordinates (the $2$-$by$-$2$ constraints). In this case, the function \( f \) must be a positive scalar multiple of a power of the Frobenius map. This, in turn, implies that \emph{no} entrywise transform can exist between two distinct LPM cones when the common sign pattern begins with two or more entries equal to \(1\).
    
\item Although this places a limitation on the existence of entrywise transforms between different LPM cones, it simultaneously highlights a structural coherence: the permitted transforms between the same LPM cones align naturally with the positive scalar multiples of the automorphisms of the underlying field, as captured by the final identities in Theorem~\ref{T:entrywise-definite}. In particular, it shows that each Cholesky map $\Psi_{\D_{\epsilon}}^{\pm 1}$ commutes with the automorphisms of the definite field $\F_q$ via the entrywise tranforms.
\end{enumerate}   
\end{remark}

We now come to the possibilities over non-definite fields parallel to Theorem~\ref{T:entrywise-definite}. As we show later, the proof of Theorem~\ref{T:entrywise-definite} is an application of a key result from \cite{guillot2024positivity}. Other main results from \cite{guillot2024positivity} similarly yield a parallel statement for non-definite fields. (However, the compatibility question does not arise here as non-definite fields do not admit a Cholesky factorization.) We mention this result for completeness under a $3$-$by$-$3$ and a $2$-$by$-$2$ constraint. Note that a slightly modified Remark~\ref{Rem:entrywise-definite}(a) applies here. 

\begin{theorem}[Entrywise transforms for non-definite fields]\label{T:entrywise-non-definite}
Suppose integers $3\leq s\leq n$ are given, and let $\F_q$ be a finite non-definite field with $q=p^k$, where $p$ is prime. Let $\epsilon,\epsilon'\in \{\pm 1\}\subseteq \R^{n\times n}$ be sign patterns satisfying $\epsilon_1=\dots=\epsilon_s=\epsilon_1'=\dots=\epsilon_s'=1$. Then for a given function $f:\F_q\to \F_q$, the following are equivalent:
\begin{enumerate}
    \item The entrywise transform $f[-]$ sends $LPM_{n}^{\F_q}(\epsilon')$ into $LPM_{n}^{\F_q}(\epsilon)$.
    \item The sign patterns are equal, and $f$ is a positive multiple of a power of the Frobenius:
    \[
    \epsilon'=\epsilon \quad \mbox{and}\quad 
    f\equiv c\cdot \mathrm{Frob}_p^{\ell} \quad \mbox{for some}\quad 0\leq \ell\leq k-1 \mbox{ and }c\in \F_q^+.
    \]
\end{enumerate}

Moreover, if we consider the special class of non-definite fields $\F_q$ of square order, then for any fixed $n\geq 2$ and $s=2$, statements (1) and (2) are equivalent.
\end{theorem}

We conclude this subsection with a question for future investigation:

\begin{question}
What are the corresponding entrywise transforms in Theorems~\ref{T:entrywise-definite} and \ref{T:entrywise-non-definite} when the $3$-$by$-$3$ and $2$-$by$-$2$ conditions, i.e., ``\( \epsilon_1 = \dots = \epsilon_s = \epsilon_1' = \dots = \epsilon_s' = 1 \)'', are removed? And how do those transforms behave with the Cholesky factorization, whenever applicable?
\end{question}

\subsection{Two group structures over the LPM cones and their compatibility with the Frobenius map}\label{subsec:group}

One of the key differences between positive definite matrices over the real/complex fields and those over a finite field is that, in the former setting, the square of a positive definite matrix \( A \), namely \( A^2 \), is again positive definite. In contrast, this property does not necessarily hold over finite fields~\cite{cooper2024positive}. In the next result, we introduce a group structure on the bigger cone \( LPM_n^{\F_q} \) (the collection of symmetric \( n \times n \) matrices over a definite field \( \F_q \) whose leading principal minors are all nonzero) under which the square of each matrix is guaranteed to be positive definite.

Recall that for a definite field \( \F_q \), the cone \( LPM_n^{\F_q} \) is the disjoint union of the sub-cones \( LPM_n^{\F_q}(\epsilon) \), indexed by sign patterns \( \epsilon \). We refer to the Cholesky factorization within each of these sub-cones, taking the representative matrix \( A_{\epsilon} = \D_{\epsilon} := \D_{\epsilon}(\pm 1) \) for \( \omega_{\pm} = \pm 1 \), as defined in~\eqref{Elpmdiag}. With this setup, we now state the group structures:

\begin{utheorem}[Group structures on the bigger LPM cone]\label{T:group-str}

Let \( n \geq 1 \) be an integer, and let \(\F_q\) be a definite finite field. Consider that a group \(({\bf L}_{n}^{\F_q^+}, \circledcirc)\) with identity element \(\I_n\) is given. Define a binary operation \(\boxdot\) on the set \(LPM_n^{\F_q}\) as follows: for any \(A, B \in LPM_n^{\F_q}\), write their unique Cholesky factorizations as
\[
A = L \D_{\epsilon} L^T \in LPM_n^{\F_q}(\epsilon), \quad \mbox{and} \quad 
B = K \D_{\epsilon'} K^T \in LPM_n^{\F_q}(\epsilon'),
\]
where \(\epsilon, \epsilon' \in \{\pm 1\}^n \subseteq \R^n\) and \(L, K \in {\bf L}_{n}^{\F_q^+}\). Then define
\begin{align}\label{eq:box-group}
A \boxdot B := (L \circledcirc K)(\D_{\epsilon} \D_{\epsilon'})(L \circledcirc K)^T.    
\end{align}
This makes \((LPM_n^{\F_q}, \boxdot)\) a group with identity element \(\I_n\), and inverse given by
\[
A = L \D_{\epsilon} L^T \mapsto L^{\circledcirc -1} \D_{\epsilon} (L^{\circledcirc -1})^T.
\]
Finally, for every \(A \in LPM_n^{\F_q}\), the square \(A^{\boxdot 2} := A \boxdot A\) is positive definite by Theorem~\ref{T:CHW}.
\end{utheorem}

\begin{proof}
It can be shown by direct verification.
\end{proof}

We present two examples of group structures on \( LPM_n^{\F_q}(\epsilon) \), one abelian and one non-abelian, obtained via the transport described in Theorem~\ref{T:group-str} from the group \({\bf L}_n^{\F_q^+}\).

\begin{example}\label{e:two-groups} Define two binary operations \(\circledcirc_1, \circledcirc_2 : {\bf L}_n^{\F_q^+} \times {\bf L}_n^{\F_q^+} \to {\bf L}_n^{\F_q^+}\) by
\[
L \circledcirc_1 K := LK, \quad \text{and} \quad L \circledcirc_2 K := \floo{L} + \floo{K} + \D(L)\D(K),
\]
where, for \(L = (l_{ij})\), the diagonal and strict lower triangular parts are defined by
\[
\D(L) := \mathrm{diag}(l_{11}, \dots, l_{nn}), \quad \text{and} \quad
\floo{L}_{ij} := \begin{cases}
l_{ij}, & i > j, \\
0, & \text{otherwise}.
\end{cases}
\]
The operation $\circledcirc_2$ was introduced in \cite{Cholesky}. Both operations endow \({\bf L}_n^{\F_q^+}\) with a group structure, with identity element \(\I_n\), and respective inverse maps given by
\[
L \mapsto L^{-1} \quad \text{for } \circledcirc_1, \qquad \text{and} \qquad L \mapsto -\floo{L} + \D(L)^{-1} \quad \text{for } \circledcirc_2.
\]
Note, the transported \(\boxdot_1\) is non-abelian; and \(\boxdot_2\) is abelian. 

Moreover, these transported groups are compatible with the Frobenius map:
\begin{align*}
\mathrm{Frob}_p[A \boxdot_j B] &= \mathrm{Frob}_p[A] \boxdot_j \mathrm{Frob}_p[B] \qquad \forall A,B\in LPM_{n}^{\F_q}, \quad \mbox{for}\quad j=1,2.
\end{align*}
\end{example}

\begin{remark}[Group structures on individual LPM cones]
In addition to the transported group structure $
\boxdot$ on the ``global'' cone $LPM_n^{\F_q}$ for definite fields $\F_q$, one can also define ``internal'' group structures $\circledast$ on each individual cone $LPM_n^{\F_q}(\epsilon)$ by modifying the operation in~\eqref{eq:box-group}:
\begin{align}\label{eq:circledas-group}
A \circledast B := (L \circledcirc K)\D_{\epsilon}(L \circledcirc K)^T    
\end{align}
where \(A, B \in LPM_n^{\F_q}(\epsilon)\) with their unique Cholesky factorizations as
\[
A = L \D_{\epsilon} L^T\quad \mbox{and} \quad
B = K \D_{\epsilon} K^T.
\] This construction endows each individual cone $LPM_n^{\F_q}(\epsilon)$, indexed by $\epsilon\in \{\pm 1\}^n\subseteq \R^n$, with a group structure. In particular, all such cones are mutually isomorphic as groups.
\end{remark}

\begin{remark}[Similar results for TPM cones]
All results (and future directions) presented above in Subsections~\ref{subsec:entrywise} and \ref{subsec:group} for the LPM cones admit direct analogues for the TPM cones as well. Since the necessary modifications are straightforward, we omit the explicit statements and proofs in the interest of brevity.
\end{remark}

\section{Proofs}

\subsection{Proof of Theorem~\ref{main-thm-1}}

Fix two notations hereafter: $[m] := \{ 1, \dots, m \}$ for all integers $m\geq 1$, and $A_{I\times J}$ for the submatrix of a matrix $A$ indexed by row set $I$ and column set $J$.

\begin{proof}[Proof of Theorem~\ref{main-thm-1}]

It is along the lines of {\cite[Theorem~A]{khare2025cholesky}}.

\begin{enumerate}[$(1)$]
    \item For a given $A_{\epsilon}\in LPM_{n}^{\F_q}(\epsilon)$ and a non-singular lower triangular $L\in \F_q^{n\times n}$, we need to show that the minor $\det (LA_{\epsilon}L^T)_{[k]\times [k]}$ has the quadratic character $\epsilon_k$. Using the Cauchy--Binet formula, we compute: 
\begin{align*}
\det (L A_\epsilon L^T)_{[k] \times [k]} = \det( L_{[k] \times [n]}
A_\epsilon (L_{[k] \times [n]})^T)
= \sum_{\substack{J,K \subseteq [n],\\ |J|=|K|=k}} \det(L_{[k]\times J})
\det(A_\epsilon)_{J\times K} \det(L_{[k]\times K})^T.
\end{align*}
Now as $L$ is lower triangular, the submatrices $L_{[k]\times J}$ are singular unless $J = [k]$. So, 
\[
\det (L A_\epsilon L^T)_{[k] \times [k]} = \ (\det L_{[k]\times [k]})^2 \det (A_\epsilon)_{[k]\times [k]}.
\]
Therefore, as $L$ is non-singular, we have the desired outcome.

\item Here $\F_q$ is a definite field or a field of even order. For convenience, consider the ``positive square root'' map \( \sqrt{\cdot} : \F_q^+ \to \F_q^+ \), where $\sqrt{a}$ is the unique $b\in \F_q^+$ such that $b^2=a$. Since $\F_q$ is definite or has even order, this map is a bijection.

The proof is based on two key observations on block matrices, and their products:
\begin{align*}
\begin{pmatrix} K & {\bf 0} \\ {\bf p}^T & s \end{pmatrix}
\begin{pmatrix} B_{\epsilon} & {\bf u} \\ {\bf u}^T & v \end{pmatrix}
\begin{pmatrix} K^T & {\bf p} \\ {\bf 0}^T & s \end{pmatrix}
&= \begin{pmatrix}
K B_{\epsilon} K^T & K (B_{\epsilon} {\bf p} + s {\bf u}) \\
({\bf p}^T B_{\epsilon} + s {\bf u}^T) K^T & {\bf p}^T B_{\epsilon} {\bf p} + v s^2 + 2 s {\bf
u}^T {\bf p}
\end{pmatrix},\nonumber\\
A_\epsilon = \begin{pmatrix} B_{\epsilon} & {\bf u} \\ {\bf u}^T & v \end{pmatrix} \in LPM_{n}^{\F_q}(\epsilon)\quad &\implies \quad B_{\epsilon}\in LPM_{n-1}^{\F_q}((\epsilon_1,\dots,\epsilon_{n-1})),\nonumber    
\end{align*}
where ${\bf u}\in \F_q^{n-1}$ and $v\in \F_q$.

Collectively, they show that for a given $A=\begin{pmatrix} B & {\bf b} \\ {\bf b}^T & c
\end{pmatrix}\in LPM_{n}^{\F_q}(\epsilon)$, to obtain the required solution $L$ for $LA_{\epsilon}L^T = A$, we first need to be able to solve for a unique lower triangular $K\in \F_q^{(n-1)\times (n-1)}$ with positive diagonal entries, such that $B=KB_{\epsilon}K^T$. Assuming this can be solved, we show that ${\bf p}\in \F_q^{n-1}$ and $s\in \F_q^+$ can be solved uniquely from this. Note that if
\begin{align*}
    {\bf b}=K(B_{\epsilon}{\bf p}+s{\bf u}) \quad \mbox{and}\quad c= {\bf p}^T B_{\epsilon} {\bf p} + v s^2 + 2 s
\end{align*}
then, from the first equation, one obtains
\begin{align*}
{\bf p} = B_{\epsilon}^{-1} (K^{-1} {\bf b} - s {\bf u})
\end{align*}
and then, up on its substitution in the second, one gets
\begin{align*}
s^2 = \frac{c - {\bf b}^T B^{-1} {\bf b}}{v - {\bf u}^T B_{\epsilon}^{-1} {\bf u}}.
\end{align*}
Now, recall the theory of Schur complements: $\det \begin{pmatrix} P
& {\bf q} \\ {\bf q}^T & t \end{pmatrix} = (\det P) (t - {\bf q}^T P^{-1}
{\bf q})$ if $P^{-1}$ exists. This gives
\begin{equation}\label{Eqsquare}
s^2 = \frac{\det A}{\det B} \cdot \frac{\det B_{\epsilon}}{\det A_{\epsilon}},
\end{equation}
in which the right side has positive quadratic character as both $A,A_{\epsilon}\in LPM_n^{\F_q}(\epsilon)$. Now take the positive square root to solve for the unique $s\in \F_q^+$. This solves for $s$ and therefore for ${\bf p}$ uniquely -- given a unique lower triangular $K\in \F_{q}^{(n-1)\times (n-1)}$ with positive diagonals.

Thus, all that remains to show is: given $B,B_{\epsilon}\in LPM_{n-1}^{\F_q}((\epsilon_1,\dots,\epsilon_{n-1}))$, can one solve for the lower triangular $K$ with positive diagonals such that $B=KB_{\epsilon} K^T$? Considering the aforementioned algorithmic process, it is sufficient to show this for $n=2$: given $a,a_{\epsilon}\in \F_q^{\pm}$, does there exists a unique $\kappa \in \F_q^+$ such that $\kappa \cdot a_{\epsilon}\cdot \kappa =\kappa^2 a_{\epsilon}=a$? Indeed, since $a/a_{\epsilon}\in \F_q^+$ -- where $\F_q$ is either a definite field or has characteristic $2$ -- it has a unique positive square root, which is our required $\kappa \in \F_q^+$.
\end{enumerate}\smallskip

\noindent Now for the the second half of the proof. \medskip

Since $\nabla A \nabla$ reverses the rows and columns of a square matrix, it interchanges the leading and trailing $k \times k$ principal minors for every $k$. Therefore it is the required linear bijection.\smallskip

For the inverse map, recall Jacobi's complementary minor formula~\cite{Jacobi}. Let \( A \in \F_q^{n \times n} \) be invertible, and let \( J, K \subseteq [n]\) of equal size \(0 < p < n\). Then:
\begin{equation}\label{Ejacobi}
\det A \cdot \det (A^{-1})_{K^c \times J^c} = (-1)^{\sum_{J}j+\sum_{K}k} \det A_{J \times K},
\end{equation}
where \( J^c := [n] \setminus J \), and similarly \( K^c \). Apply this identity for \( J = K = [k] \) for \( 0 < k < n \). If \( A \in LPM_n^{\F_q}(\epsilon) \), then the trailing principal minors of \( A^{-1} \) satisfy:
\[
\det (A^{-1})_{[k]^c \times [k]^c} = \frac{\det A_{[k] \times [k]}}{\det A}.
\]
Therefore, the quadratic character of this ratio, and hence of the trailing principal minor, is \( \epsilon_k \epsilon_n \). If \( k = n \), we have \( \det A^{-1} = 1 / \det A \), which has the character \( \epsilon_n \).

Moreover, starting with \( A \in TPM_n^{\F_q}(\epsilon') \), one may apply~\eqref{Ejacobi} with \( J = K = [k]^c \) instead, ensuring that the map is a bijection. 

The final steps of showing $(1)$ and $(2)$ for the TPM cones are a direct applications of the assertions proved above for the linear and non-linear bijections.
\end{proof}

\begin{remark}
The only step in the Proof of Theorem~\ref{main-thm-1}(2) that does not work out for non-definite fields is the last paragraph of (2). More precisely, if $\F_q$ is non-definite, then for each $a\in \F_q^{+}$, there exist distinct $b_1,b_2\in \F_q^{\pm}$ such that $b_j^2=a$. Therefore, the ``positive square root'' function is not well-defined here, and as a result a Cholesky factorization does not exist for non-definite fields.
\end{remark}

\subsection{Proof of Theorems~\ref{T:entrywise-definite} and \ref{T:entrywise-non-definite}}

Here the proofs are applications of the results in \cite{guillot2024positivity} in the context of positivity preservers.

\begin{proof}[Proof of Theorem~\ref{T:entrywise-definite} $(1)\iff (2)$] The implication $(2) \implies (1)$ follows from {\cite[Proposition~2.12]{guillot2024positivity}}. For the other implication, suppose $A'\in LPM_{s}^{\F_q}((1,1,\dots,1))$. Then $A:=A'\oplus \D_{\epsilon''}(\omega_{\pm})\in LPM_n^{\F_q}(\epsilon')$, for some $\omega_{\pm}\in \F_q^{\pm}$, where $\epsilon''=(\epsilon_{s+1},\dots,\epsilon_n)$. Therefore $f[A]\in LPM_{n}^{\F_q}(\epsilon)$, and in particular, $f[A']=f[A]_{[s]\times [s]}\in LPM_{s}^{\F_q}((1,1,\dots,1))$. This means that $f[-]$ is an entrywise positivity preserver over $s\times s$ positive definite matrices, where $\F_q$ is a finite definite field. Therefore {\cite[Theorem~B]{guillot2024positivity}} implies that $f\equiv c\cdot \mathrm{Frob}^{\ell}$ for some $c\in \F_q^+$ and $\ell\in 0,1,\dots,k-1$. Moreover, it follows from {\cite[Proposition~2.12]{guillot2024positivity}} that $\epsilon'=\epsilon$.
\end{proof}

The proof of Theorem~\ref{T:entrywise-non-definite} is similarly completed, using {\cite[Theorem C]{guillot2024positivity}} instead.

The final identities in Theorem~\ref{T:entrywise-definite} and the final assertion in Example~\ref{e:two-groups} require the next:

\begin{lemma}\label{lemma:frob-usual-prod}
Let $\F_q$ be a finite field with $q=p^k$ for $p$ a prime. Suppose $A$ and $B$ are matrices with entries in $\F_q$, such that $AB$ is well defined. Then $\mathrm{Frob}_p[AB]=\mathrm{Frob}_p[A]\mathrm{Frob}_p[B]$.
\end{lemma}
\begin{proof}
We have the following as $\mathrm{Frob}_p$ preserves the field operations: 
\begin{align*}
    \mathrm{Frob}_p((AB)_{ij}) = \mathrm{Frob}_p(\sum_{k}a_{ik}b_{kj}) = \sum_{k}\mathrm{Frob}_p(a_{ik})\mathrm{Frob}_p(b_{kj}) = (\mathrm{Frob}_p[A]\mathrm{Frob}_p[B])_{ij}.
\end{align*}
This completes the proof.
\end{proof}

\begin{proof}[Proof of Theorem~\ref{T:entrywise-definite} (compatibility with the Frobenius map)] It follows precisely due to Lemma~\ref{lemma:frob-usual-prod} and the fact that $\Psi_{A_{\epsilon}}^{\pm 1}$ are bijections. Moreover, when we restrict to $A_{\epsilon}=\D_{\epsilon}$, then $\mathrm{Frob}_{p}[\D_{\epsilon}]= \D_{\epsilon}$. This yields the final identities on commutativity.
\end{proof}

\section*{Acknowledgments}

I am thankful for the support by the Centre de recherches math\'ematiques and Universit\'e Laval and their CRM-Laval Postdoctoral Fellowship, and the Alliance grant. I would also like to sincerely thank Apoorva Khare for comments and discussions.




\begin{thebibliography}{18}

\bibitem{Benoit}
Benoit.
\newblock {\em Note sur une m\'ethode de r\'esolution des \'equations
normales provenant de l'application de la m\'ethode des moindres carrés
\`a un syst\`eme d'\'equations lin\'eaires en nombre inf\'erieur \`a
celui des inconnues (Proc\'ed\'e du Commandant Cholesky)}.
\newblock \href{https://doi.org/10.1007/BF03031308}{Bull.\
g\'eod\'esique}, I.\ -- Notices Scientifiques, 2:67--77, 1924.

\bibitem{cooper2024positive}
Joshua Cooper, Erin Hanna, and Hays Whitlatch.
\newblock {\em Positive-definite matrices over finite fields}.
\newblock \href{https://doi.org/10.1216/rmj.2024.54.423}{Rocky Mountain J.\ Math.}, 54(2): 423--438, 2024.




\bibitem{guillot2024positivity}
Dominique Guillot, Himanshu Gupta, Prateek Kumar Vishwakarma, Chi Hoi Yip. \newblock {\em Positivity preservers over finite fields}. \newblock \href{https://doi.org/10.1016/j.jalgebra.2025.07.016}{J.\ Algebra}, 684:479--523, 2025. 


\bibitem{guillot2024positivity-fpsac}
Dominique Guillot, Himanshu Gupta, Prateek Kumar Vishwakarma, Chi Hoi Yip. \newblock {\em Entrywise transforms and positive definite matrices over finite fields}.
\newblock {37th International Conference on Formal Power Series
and Algebraic Combinatorics (FPSAC 2025)}, \href{https://www.mat.univie.ac.at/~slc/wpapers/FPSAC2025/}{S\'em.\ Lothar.\ Combin.} 93B (2025), Article \#69, 12 pp.



\bibitem{Jacobi}
Carl Gustav Jacob Jacobi.
\newblock {\em De formatione et proprietatibus Determinatium}.
\newblock \href{https://doi.org/10.1515/crll.1841.22.285}{J.\ reine
angew.\ Math.}, 22:285--318, 1841.



\bibitem{khare-book}
Apoorva Khare.
\newblock {\em Matrix Analysis and Entrywise Positivity Preservers.}
\newblock London Mathematical Society Lecture
Note Series, Cambridge University Press, 2022.

\bibitem{khare2025cholesky}
Apoorva Khare and Prateek Kumar Vishwakarma.
\newblock {\em Cholesky decomposition for symmetric matrices, Riemannian geometry, and random matrices}.
\newblock \href{https://arxiv.org/abs/2508.02715}{arXiv:2508.02715}, 2025.


\bibitem{Cholesky}
Zhenhua Lin.
\newblock {\em Riemannian geometry of Symmetric Positive Definite
matrices via Cholesky decomposition}.
\newblock \href{https://doi.org/10.1137/18M1221084}{SIAM J.\ Matrix
Anal.\ Appl.}, 40(4):1353--1370, 2019.

\bibitem{polya-szego}
Georg P\'olya and Gabor Szeg\H{o}.
\newblock {\em Aufgaben und Lehrs\"atze aus der Analysis. Band II: Funktionentheorie, Nullstellen, Polynome Determinanten, Zahlentheorie, volume Band 74 of Heidelberger Taschenb\"ucher [Heidelberg Paperbacks]}. \newblock {Springer-Verlag, Berlin-New York, 1971}.




\bibitem{rudin1959positive}
Walter Rudin.
\newblock {\em Positive definite sequences and absolutely monotonic functions}.
\newblock \href{https://doi.org/10.1215/S0012-7094-59-02659-6}{Duke Math. J.}, 26:617--622, 1959.


\bibitem{schoenberg1942positive}
Isaac Jacob Schoenberg.
\newblock {\em Positive definite functions on spheres}.
\newblock \href{https://doi.org/10.1215/S0012-7094-42-00908-6}{Duke Math. J.}, 9:96--108, 1942.

\bibitem{schur1911}
Issai Schur.
\newblock {\em Bemerkungen zur Theorie der beschr\"ankten Bilinearformen mit unendlich vielen Ver\"anderlichen}.
\newblock \href{http://eudml.org/doc/149352}{J. reine angew. Math.}, 140:1--28, 1911.


\bibitem{stanley}
Richard Peter Stanley.
\newblock {\em Enumerative Combinatorics. Vol.\ 1, 2nd Ed.}
\newblock Cambridge University Press, 2012.





\end{thebibliography}
\end{document}